\pdfoutput=1
\RequirePackage{ifpdf}
\ifpdf 
\documentclass[pdftex]{sigma}
\else
\documentclass{sigma}
\fi

\numberwithin{equation}{section}

\newtheorem{Theorem}{Theorem}[section]
\newtheorem{Corollary}[Theorem]{Corollary}
\newtheorem{Lemma}[Theorem]{Lemma}
\newtheorem{Proposition}[Theorem]{Proposition}
 { \theoremstyle{definition}
\newtheorem{Definition}[Theorem]{Definition}
\newtheorem{Example}[Theorem]{Example}
\newtheorem{Remark}[Theorem]{Remark} }

\begin{document}


\newcommand{\arXivNumber}{1603.09569}

\renewcommand{\PaperNumber}{086}

\FirstPageHeading

\ShortArticleName{On Jacobi Inversion Formulae for Telescopic Curves}

\ArticleName{On Jacobi Inversion Formulae for Telescopic Curves}

\Author{Takanori AYANO}

\AuthorNameForHeading{T.~Ayano}

\Address{Osaka City University, Advanced Mathematical Institute,\\ 3-3-138 Sugimoto, Sumiyoshi-ku, Osaka, 558-8585, Japan}
\Email{\href{mailto:tayano7150@gmail.com}{tayano7150@gmail.com}}

\ArticleDates{Received May 06, 2016, in f\/inal form August 23, 2016; Published online August 27, 2016}

\Abstract{For a hyperelliptic curve of genus $g$, it is well known that the symmetric products of $g$ points on the curve are expressed in terms of their Abel--Jacobi image by the hyper\-elliptic sigma function (Jacobi inversion formulae). Matsutani and Previato gave a natural generalization of the formulae to the more general algebraic curves def\/ined by $y^r=f(x)$, which are special cases of $(n,s)$ curves, and derived new vanishing properties of the sigma function of the curves $y^r=f(x)$. In this paper we extend the formulae to the telescopic curves proposed by Miura and derive new vanishing properties of the sigma function of telescopic curves. The telescopic curves contain the $(n,s)$ curves as special cases.}

\Keywords{sigma function; inversion of algebraic integrals; vanishing of sigma function; Riemann surface; telescopic curve}

\Classification{14H42; 14H50; 14H55}

\section{Introduction}

The theory of the elliptic function was the one of the main subjects of the research of mathematics in the nineteenth century. Now the beautiful theory of the elliptic function is constructed and is applied to many f\/ields such as mathematical physics, integrable system, number theory, engineering, and cryptography. In integrable system, it is well known that the elliptic function gives an exact solution of some nonlinear dif\/ferential equations. In cryptography, the cryptosystem using the elliptic curves is used widely. Recently, with the scientif\/ic development, we have to analyze many complicated nonlinear phenomena and it is necessary to give exact solutions of many nonlinear dif\/ferential equations in order to analyze the phenomena precisely. In cryptography, it is necessary to make a wider class of algebraic curves available to the cryptosystem for assuring the safety of cryptosystem. Therefore it is very important to construct the basic theory of the Abelian function, which is a generalization of the elliptic function to several variables. The sigma function plays an important role in the theory of the Abelian function.

The multivariate sigma function is introduced by F.~Klein \cite{Kl1,Kl2} for hyperelliptic curves as a generalization of the Weierstrass's elliptic sigma function. Recently, the hyperelliptic sigma function is generalized to the more general plane algebraic curves called $(n,s)$ curves \cite{BEL1,BEL2,BEL3,EEL,N1}. The sigma function is obtained by modifying Riemann's theta function so as to be modular invariant, i.e., it does not depend on the choice of a canonical homology basis. Further the sigma function has some remarkable algebraic properties that it is directly related with the def\/ining equations of an algebraic curve. From these algebraic properties, the sigma function is expected to have many applications in mathematical physics etc.~\cite{BEL3}. Further the sigma function is useful to describe a solution of the inversion problem of algebraic integrals. The Jacobi inversion problem for hyperelliptic curves is described as follows.

Let $X$ be a hyperelliptic curve of genus $g$ def\/ined by $y^2=f(x)$,
\begin{gather*}
f(x)=x^{2g+1}+\lambda_{2g}x^{2g}+\cdots+\lambda_1x+\lambda_0,\qquad \lambda_i\in\mathbb{C}.
\end{gather*}
Let $du_i=-\frac{x^{g-i}}{2y}dx$, $1\le i\le g$, be the holomorphic one forms on $X$ and $du={}^t(du_1,\dots,du_g)$. For $1\le k\le g$, $P_1,\dots,P_k\in X\backslash\infty$, and $u^{[k]}=\sum\limits_{i=1}^k\int_{\infty}^{P_i}du$, one wants to express the coordinates of~$P_i$ in terms of~$u^{[k]}$.

For $k=g$ and $P_i=(x_i,y_i)\in X$, we def\/ine the symmetric polynomial $e_i$ by
\begin{gather*}
e_i=\sum_{1\le\ell_1<\cdots<\ell_i\le g}x_{\ell_1}\cdots x_{\ell_i}.
\end{gather*}
Let $\sigma(u)$ be the sigma function of $X$ and $S^{g}(X)$ the $g$-th symmetric products of $X$. Then the following theorem is well-known \cite{Baker}.

\begin{theorem*}[Jacobi inversion formulae] If $\sum\limits_{i=1}^gP_i\in S^g(X\backslash\infty)$ is a general divisor, then we have
\begin{gather*}
\wp_{1,i}\big(u^{[g]}\big)=(-1)^{i-1}e_i,\qquad 1\le i\le g,
\end{gather*}
where $\wp_{i,j}(u)=-\frac{\partial^2}{\partial u_i\partial u_j}\log\sigma(u)$.
\end{theorem*}

The inversion of algebraic integrals is deeply related to the problems of mathematical physics (cf.~\cite{E2, E1}).

Matsutani and Previato \cite{Matsutani} gave a natural generalization of the above formulae for any $1\le k\le g$ and the more general plane algebraic curves def\/ined by
\begin{gather}
y^r=x^s+\lambda_{s-1}x^{s-1}+\cdots+\lambda_0,\label{yr}
\end{gather}
where $r$ and $s$ are relatively prime positive integers and $\lambda_i\in\mathbb{C}$. These curves are special cases of the $(n,s)$ curves. Furthermore, in~\cite{Matsutani}, new vanishing properties of the sigma function of the curves def\/ined by (\ref{yr}) are derived by using the extended Jacobi inversion formulae.

On the other hand, in \cite{Miu}, Miura introduced a certain canonical form, Miura canonical form, for def\/ining equations of any non-singular algebraic curve. A telescopic curve \cite{Miu} is a~special curve for which Miura canonical form is easy to determine. Let $m\geq 2$ and $(a_1,\dots ,a_m)$ a~sequence of relatively prime positive integers satisfying certain condition. Then the telescopic curve associated with $(a_1,\dots ,a_m)$ or the $(a_1,\dots ,a_m)$ curve is the algebraic curve def\/ined by certain $m-1$ equations in ${\mathbb C}^m$. For $m=2$, the telescopic curves are equal to the $(n,s)$ curves.

In this paper we extend the formulae obtained in \cite{Matsutani} to the telescopic curves (Theorems~\ref{main2} and~\ref{main3}). More specif\/ically, for the telescopic curves, we give formulae which express the $\wp$-function and the ratio of the derivative of the sigma function by the ratio of the determinants of certain matrices consisting of the algebraic functions. Under a certain condition, a coordinate of one point on the telescopic curves can be expressed in terms of its Abel--Jacobi image by the derivatives of the sigma function (Corollary~\ref{999}). Furthermore we derive new vanishing properties of the sigma function of the telescopic curves as a corollary of the formulae (Corollaries~\ref{9} and~\ref{55}). Finally we comment that the Jacobi inversion formulae are derived for $(3,4,5)$ curves in~\cite{Matsutani3} and $(3,7,8)$, $(6,13,14,15,16)$ curves in~\cite{Matsutani4}, which are not telescopic.

The present paper is organized as follows. In Section~\ref{section2}, the def\/inition of the telescopic curves is given. In Section~\ref{section3}, the fundamental dif\/ferential of second kind for the telescopic curves is reviewed and a coef\/f\/icient of the second kind dif\/ferentials is determined explicitly. In Section~\ref{section4}, the def\/inition of the sigma function of telescopic curves and the expression of the fundamental dif\/ferential of second kind by the sigma function are given. In Section~\ref{section5}, Frobenius--Stickelberger matrix is def\/ined. In Section~\ref{section6}, Riemann's singularity theorem is reviewed. In Section~\ref{section7}, a~generalization of Jacobi inversion formulae to telescopic curves is given. In Section~\ref{section8}, as an example, the formulae for the $(4,6,5)$ curves are given. In Section~\ref{section9}, some new vanishing properties of the sigma function of telescopic curves are given. In Section~\ref{section10}, as an example, the vanishing properties of the sigma function of the $(4,6,5)$ curves are given.

\section{Telescopic curves}\label{section2}

In this section we brief\/ly review the def\/inition of telescopic curves following \cite{Aya1, Miu}.

For $m\geq 2$, let $(a_1,\dots ,a_m)$ be a sequence of positive integers such that $\operatorname{gcd}(a_1,\dots ,a_m)=1$, $a_i\ge2$ for any $i$, and
\begin{gather*}
\frac{a_i}{d_i}\in \frac{a_1}{d_{i-1}}{\mathbb Z}_{\geq 0}+ \cdots+\frac{a_{i-1}}{d_{i-1}}{\mathbb Z}_{\geq 0},
\qquad 2\leq i\leq m,
\end{gather*}
where $d_i=\operatorname{gcd}(a_1,\dots ,a_i)$.

Let
\begin{gather*}
B(A_m)=\left\{(\ell_1,\dots ,\ell_m)\in {\mathbb Z}_{\geq 0}^m \,\big|\, 0\leq \ell_i\leq \frac{d_{i-1}}{d_i}-1 \ \text{for} \ 2\leq i\leq m\right\}.
\end{gather*}

\begin{Lemma}[\cite{Aya1, Miu}]\label{lem-2-1} For any $a\in a_1\mathbb{Z}_{\ge0}+\cdots+a_m\mathbb{Z}_{\ge0}$, there exists a unique element $(k_1,\dots,k_m)$ of $B(A_m)$ such that
\begin{gather*}
\sum_{i=1}^ma_ik_i=a.
\end{gather*}
\end{Lemma}

By this lemma, for any $2\leq i\leq m$, there exists a unique sequence $(\ell_{i,1},\dots ,\ell_{i,m})\in B(A_m)$ satisfying
\begin{gather*}
\sum_{j=1}^m a_j \ell_{i,j}= a_i\frac{d_{i-1}}{d_i}.
\end{gather*}

\begin{Lemma}[\cite{Aya2}]\label{defining}
For any $2\le i\le m$, we have $\ell_{i,j}=0$ for $j\geq i$.
\end{Lemma}

Consider $m-1$ polynomials in $m$ variables $x_1,\dots ,x_m$ given by
\begin{gather}
F_i(x)=x_i^{d_{i-1}/d_i}-\prod_{j=1}^{i-1} x_j^{\ell_{i,j}}- \sum \lambda^{(i)}_{j_1,\dots ,j_m}x_1^{j_1}\cdots x_m^{j_m},
\qquad 2\leq i\leq m,\label{eq-2-5}
\end{gather}
where $\lambda^{(i)}_{j_1,\dots ,j_m}\in\mathbb{C}$ and the sum of the right-hand side is over all $(j_1,\dots ,j_m)\in B(A_m)$ such that
\begin{gather*}
\sum_{k=1}^m a_kj_k<a_i\frac{d_{i-1}}{d_i}.
\end{gather*}

Let $X^{\rm{af\/f}}$ be the common zeros of $F_2$,\dots ,$F_ m$:
\begin{gather*}
X^{\rm{af\/f}}=\big\{(x_1,\dots ,x_m)\in\mathbb{C}^m\,|\,F_i(x_1,\dots ,x_m)=0,\, 2\leq i\leq m\big\}.
\end{gather*}
In \cite{Aya1, Miu}, $X^{\rm{af\/f}}$ is proved to be an af\/f\/ine algebraic curve. We assume that $X^{\rm{af\/f}}$ is nonsingular. Let $X$ be the compact Riemann surface corresponding to $X^{\rm{af\/f}}$. Then $X$ is obtained from $X^{\rm{af\/f}}$ by adding one point, say $\infty$ \cite{Aya1, Miu}. The genus of $X$ is given by \cite{Aya1, Miu}
\begin{gather}
g=\frac{1}{2}\left\{ 1-a_1+\sum_{i=2}^m\left(\frac{d_{i-1}}{d_i}-1\right)a_i \right\}.\label{eq-2-6}
\end{gather}
We call $X$ the telescopic curve associated with $(a_1,\dots ,a_m)$. The numbers $a_1,\dots ,a_m$ are a~ge\-ne\-rator of the semigroup of non-gaps at~$\infty$.

\begin{Example}\quad
\begin{enumerate}\itemsep=0pt
\item[(i)] The telescopic curve associated with a pair of relatively prime integers $(n,s)$ is the $(n,s)$ curve introduced in~\cite{BEL2}.

\item[(ii)] For $A_3=(4,6,5)$, polynomials $F_i$ are given by
\begin{gather*}
F_2(x)=x_2^2-x_1^3-\lambda^{(2)}_{0,1,1}x_2x_3-\lambda^{(2)}_{1,1,0}x_1x_2-\lambda^{(2)}_{1,0,1}x_1x_3-\lambda^{(2)}_{2,0,0}x_1^2
-\lambda^{(2)}_{0,1,0}x_2 \\
\hphantom{F_2(x)=}{}-\lambda^{(2)}_{0,0,1}x_3-\lambda^{(2)}_{1,0,0}x_1-\lambda^{(2)}_{0,0,0},
\\
F_3(x)=x_3^2-x_1x_2-\lambda^{(3)}_{1,0,1}x_1x_3-\lambda^{(3)}_{2,0,0}x_1^2-\lambda^{(3)}_{0,1,0}x_2-\lambda^{(3)}_{0,0,1}x_3
-\lambda^{(3)}_{1,0,0}x_1-\lambda^{(3)}_{0,0,0}.
\end{gather*}
\end{enumerate}
\end{Example}

For a meromorphic function $f$ on $X$, we denote by $\operatorname{ord}_{\infty}(f)$ the order of a pole at $\infty$. Then we have $\operatorname{ord}_{\infty}(x_i)=a_i$. We enumerate the monomials $x_1^{\alpha_1}\cdots x_m^{\alpha_m}$, $(\alpha_1,\dots,\alpha_m)\in B(A_m)$, according as the order of a pole at $\infty$ and denote them by $\varphi_i$, $i\geq 1$. In particular we have $\varphi_1=1$. The set $\{\varphi_i\}_{i=1}^{\infty}$ is a basis of meromorphic functions on $X$ with a pole only at $\infty$.

Let $G$ be the $(m-1)\times m$ matrix def\/ined by
\begin{gather*}
G=\left(\frac{\partial F_i}{\partial x_j}\right)_{2\leq i\leq m, \, 1\le j\le m}
\end{gather*}
and $G_k$ the $(m-1)\times (m-1)$ matrix obtained by deleting the $k$-th column from $G$. Then a~basis of holomorphic one forms is given by
\begin{gather*}
du_i=-\frac{\varphi_{g+1-i}}{\det G_1}dx_1,\qquad 1\le i\le g.
\end{gather*}

Let $(w_1,\dots ,w_g)$ be the gap sequence at $\infty$:
\begin{gather*}
\{w_i\,|\,1\leq i\leq g\}={\mathbb Z}_{\geq0} \backslash\left\{\sum_{i=1}^m a_i{\mathbb Z}_{\geq0}\right\}, \qquad w_1<\cdots<w_g.
\end{gather*}
In particular $w_1=1$, since $g\geq 1$. The following lemma is proved in \cite{Aya1}.

\begin{Lemma}\label{lem-2-2}
We have $w_{g}=2g-1$. In particular, $du_g$ has a zero of order $2g-2$ at $\infty$.
\end{Lemma}

From Lemma \ref{lem-2-2}, we f\/ind that the vector of Riemann constants for a telescopic curve with a base point $\infty$ is a half-period.

\begin{Lemma}[\cite{Aya2}]\label{b}
It is possible to take a local parameter $z$ around $\infty$ such that
\begin{gather}
x_1=\frac{1}{z^{a_1}}, \qquad x_i=\frac{1}{z^{a_i}}(1+O(z)),\qquad 2\leq i\leq m. \label{eq-3-1}
\end{gather}
\end{Lemma}

\begin{Proposition}[\cite{Aya2}]\label{c}
For $1\leq i\leq g$, the expansion of $du_i$ at $\infty$ is of the form
\begin{gather*}
du_i=z^{w_i-1}(1+O(z))dz.
\end{gather*}
\end{Proposition}

For the telescopic curve $X$ associated with $A_m=(a_1,\dots,a_m)$, we def\/ine the partition by
\begin{gather*}
\mu(A_m)=(w_g,\dots,w_1)-(g-1,\dots,0).
\end{gather*}

\begin{Proposition}[\cite{BEL2,N1}]
The Young diagram of $\mu(A_m)$ is symmetric.
\end{Proposition}

\begin{Example}
For $(4,6,5)$ curves, we have $g=4$, $w_1=1$, $w_2=2$, $w_3=3$, $w_4=7$ and $\mu((4,6,5))=(4,1,1,1)$. For $(4,6,7)$ curves, we have $g=5$, $w_1=1$, $w_2=2$, $w_3=3$, $w_4=5$, $w_5=9$ and $\mu((4,6,7))=(5,2,1,1,1)$. Therefore the Young diagrams of $(4,6,5)$ curves and $(4,6,7)$ curves are as follows.

\smallskip

{\normalsize \hspace{14ex}$(4,6,5)$ curves\hspace{23ex}$(4,6,7)$ curves}

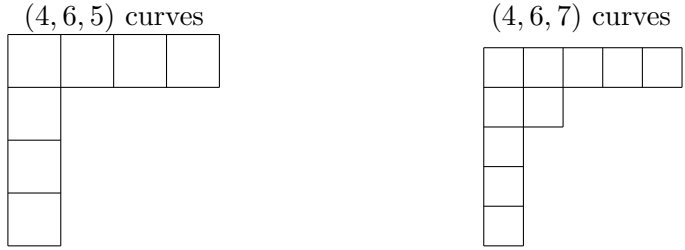
\begin{picture}(400,80)
\put(60,0){\line(1,0){20}}
\put(60,20){\line(1,0){20}}
\put(60,40){\line(1,0){20}}
\put(60,80){\line(1,0){80}}
\put(60,60){\line(1,0){80}}
\put(60,0){\line(0,1){80}}
\put(80,0){\line(0,1){80}}
\put(100,60){\line(0,1){20}}
\put(120,60){\line(0,1){20}}
\put(140,60){\line(0,1){20}}

\put(240,0){\line(1,0){15}}
\put(240,15){\line(1,0){15}}
\put(240,30){\line(1,0){15}}
\put(240,45){\line(1,0){30}}
\put(240,60){\line(1,0){75}}
\put(240,75){\line(1,0){75}}
\put(240,0){\line(0,1){75}}
\put(255,0){\line(0,1){75}}
\put(270,45){\line(0,1){30}}
\put(285,60){\line(0,1){15}}
\put(300,60){\line(0,1){15}}
\put(315,60){\line(0,1){15}}
\end{picture}
\end{Example}

\section{Fundamental dif\/ferential of second kind}\label{section3}

A fundamental dif\/ferential of second kind plays an important role in the theory of the sigma function. We recall its def\/inition.

\begin{Definition}
A two form $\omega(P,Q)$ on $X\times X$ is called a fundamental dif\/ferential of second kind if the following conditions are satisf\/ied:
\begin{enumerate}\itemsep=0pt
\item[(i)] $\omega(P,Q)=\omega(Q,P)$,

\item[(ii)] $\omega(P,Q)$ is holomorphic except $\{(R,R)\,|\,R\in X\}$ where it has a double pole,

\item[(iii)] for $R\in X$, take a local coordinate $t$ around $R$, then the expansion around $(R,R)$ is of the form
\begin{gather*}
\omega(P,Q)=\left(\frac{1}{(t_P-t_Q)^2}+\text{regular}\right)dt_Pdt_Q.
\end{gather*}
\end{enumerate}
\end{Definition}
A fundamental dif\/ferential of second kind exists but is not unique. Let $\omega_1(P,Q)$ be a fundamental dif\/ferential of second kind. Then a two form $\omega_2(P,Q)$ is a fundamental dif\/ferential of second kind if and only if there exists $\{c_{ij}\}_{i,j=1,\dots,g}\in\mathbb{C}$ such that $c_{ij}=c_{ji}$ and
\begin{gather*}
\omega_2(P,Q)=\omega_1(P,Q)+\sum_{i,j=1}^gc_{ij}dv_i(P)dv_j(Q),
\end{gather*}
where $\{dv_i\}_{i=1}^g$ is a basis of holomorphic one forms on $X$.

For a telescopic curve $X$, a fundamental dif\/ferential of second kind is algebraically constructed in~\cite{Aya1}. We recall its construction. Note that the construction inherits all steps of classical construction in~\cite{Baker} that was recently recapitulated and generalized in~\cite{EEL,N1} for the $(n,s)$ curves.

Let $X$ be a telescopic curve of genus $g$. We def\/ine the 2-form $\widehat{\omega}(P,Q)$ on $X\times X$ by
\begin{gather*}
\widehat{\omega}(P,Q)=d_Q\Omega(P,Q)+ \sum_{i=1}^g du_i(P)dr_i(Q),
\end{gather*}
where $P=(x_1,\dots ,x_m)$, $Q=(y_1,\dots ,y_m)$ are points on $X$,
\begin{gather*}
\Omega(P,Q)=\frac{\det H(P,Q)}{(x_1-y_1)\det G_1(P)}dx_1,
\end{gather*}
$H=(h_{ij})_{2\leq i,j\leq m}$ with
\begin{gather*}
h_{ij}=\frac{F_i(y_1,\dots ,y_{j-1},x_j,x_{j+1},\dots ,x_m)-F_i(y_1,\dots ,y_{j-1},y_j,x_{j+1},\dots ,x_m)}{x_j-y_j},
\end{gather*}
and $dr_i$ is a second kind dif\/ferential with a pole only at $\infty$. The set
\begin{gather*}
\left\{\frac{\varphi_i(P)}{\det G_1(P)}dx_1\right\}_{i=1}^{\infty}
\end{gather*}
is a basis of meromorphic one forms on $X$ with a pole only at $\infty$ \cite{Aya1, N1}. It is possible to take $\{dr_i\}_{i=1}^g$ such that $\widehat{\omega}(P,Q)=\widehat{\omega}(Q,P)$ \cite{Aya1, N1}. If we take $\{dr_i\}_{i=1}^g$ such that $\widehat{\omega}(P,Q)=\widehat{\omega}(Q,P)$, then $\widehat{\omega}(P,Q)$ becomes a fundamental dif\/ferential of second kind \cite{Aya1, N1}.

We assign degrees as
\begin{gather*}
\deg x_k=\deg y_k=a_k,\qquad \deg\lambda_{j_1,\dots,j_m}^{(i)}=a_id_{i-1}/d_i-\sum_{k=1}^ma_kj_k.
\end{gather*}

\begin{Lemma}\label{167}\quad
\begin{enumerate}\itemsep=0pt
\item[$(i)$] For $1\le k\le m$, $\det G_k(Q)$ is homogeneous of degree $\sum\limits_{i=2}^ma_id_{i-1}/d_i-\sum\limits_{i=1}^ma_i+a_k$ with respect to the coefficients $\big\{\lambda_{j_1,\dots,j_m}^{(i)}\big\}$ and the variables $y_1,\dots,y_m$.

\item[$(ii)$] $\det H(P,Q)$ is homogeneous of degree $\sum\limits_{i=2}^ma_i(d_{i-1}/d_i-1)$ with respect to the coefficients $\big\{\lambda_{j_1,\dots,j_m}^{(i)}\big\}$ and the variables $x_1,\dots,x_m,y_1,\dots,y_m$.
\end{enumerate}
\end{Lemma}

\begin{proof}
For $2\le i\le m$ and $1\le j\le m$, $\frac{\partial F_i(y)}{\partial y_j}$ is homogeneous of degree $a_id_{i-1}/d_i-a_j$ with respect to $\big\{\lambda_{j_1,\dots,j_m}^{(i)}\big\}$ and $y_1,\dots,y_m$. Therefore we obtain~(i). For $2\le i,j\le m$, $h_{ij}$~is homogeneous of degree $a_id_{i-1}/d_i-a_j$ with respect to $\big\{\lambda_{j_1,\dots,j_m}^{(i)}\big\}$ and $x_1,\dots,x_m, y_1,\dots,y_m$. Therefore we obtain~(ii).
\end{proof}

We have
\begin{gather*}
d_Q\Omega(P,Q)
=\frac{\!\Big\{\sum\limits_{i=1}^m ({-}1)^{i+1}(x_1\!-\!y_1)\frac{\partial \det H}{\partial y_i}(P,Q)\det G_i(Q)\Big\}\!+\!\det G_1(Q)\det H(P,Q)\!}{(x_1-y_1)^2\det G_1(P)\det G_1(Q)}dx_1dy_1,
\end{gather*}
where the numerator is homogeneous of degree $2\sum\limits_{i=2}^m(d_{i-1}/d_i-1)a_i$ with respect to the coef\/f\/icients $\big\{\lambda_{j_1,\dots,j_m}^{(i)}\big\}$ and the variables $x_1,\dots,x_m, y_1,\dots,y_m$~\cite{Aya1}. We have $\operatorname{ord}_{\infty}(\varphi_g)=2g-2$ and $\operatorname{ord}_{\infty}(\varphi_{g+1})=2g$ from Lemma~\ref{lem-2-2}. Therefore, from~(\ref{eq-2-6}), we have $\operatorname{ord}_{\infty}(\varphi_g\varphi_{g+1})=4g-2=2\sum\limits_{i=2}^m(d_{i-1}/d_i-1)a_i-2a_1$. Since
\begin{gather*}
du_1(P)=-\frac{\varphi_g(P)}{\det G_1(P)}dx_1,
\end{gather*}
if we take $\{dr_i\}_{i=1}^g$ such that $\widehat{\omega}(P,Q)=\widehat{\omega}(Q,P)$, then we f\/ind that $dr_1$ has the following form
\begin{gather}
dr_1(Q)=\sum_{i=1}^{g+1}c_i\frac{\varphi_i(Q)}{\det G_1(Q)}dy_1,\qquad c_i\in\mathbb{C}.
\label{1235}
\end{gather}

Let{\samepage
\begin{gather*}
\sum_{i=1}^gdu_i(P)dr_i(Q)=\frac{\sum c_{i_1,\dots,i_m;j_1,\dots,j_m}x_1^{i_1}\cdots x_m^{i_m}y_1^{j_1}\cdots y_m^{j_m}}{\det G_1(P)\det G_1(Q)}dx_1dy_1,
\end{gather*}
where $(i_1,\dots,i_m), (j_1,\dots,j_m)\in B(A_m)$ and $c_{i_1,\dots,i_m;j_1,\dots,j_m}\in\mathbb{C}$.}

We want to determine the coef\/f\/icients $c_{i_1,\dots,i_m;j_1,\dots,j_m}$ such that $\widehat{\omega}(P,Q)=\widehat{\omega}(Q,P)$ explicitly. For $(n,s)$ curves, i.e., $m=2$, such coef\/f\/icients are determined explicitly in~\cite{Suzuki}. Let $\varphi_g(P)=x_1^{k_1}\cdots x_m^{k_m}$ and $\varphi_{g+1}(Q)=y_1^{\ell_1}\cdots y_m^{\ell_m}$, where $(k_1,\dots,k_m), (\ell_1,\dots,\ell_m)\in B(A_m)$. In this paper, in order to derive the Jacobi inversion formulae for telescopic curves, we determine the coef\/f\/icient $c_{k_1,\dots,k_m;\ell_1,\dots,\ell_m}$ for telescopic curves.

\begin{Proposition}\label{111}
We have
\begin{gather*}
c_{k_1,\dots,k_m,\ell_1,\dots,\ell_m}=1.
\end{gather*}
\end{Proposition}

In order to prove Proposition \ref{111}, we need some lemmas.

\begin{Lemma}\label{lemma5}For $1\le k\le m$, we have
\begin{gather*}
\det G_k(Q)=(-1)^{k+1}a_ky_1^{\gamma_1}\cdots y_m^{\gamma_m}+\sum\alpha_{i_1,\dots,i_m}y_1^{i_1}\cdots y_m^{i_m},
\end{gather*}
where $(\gamma_1,\dots,\gamma_m)$ is the unique element of $B(A_m)$ such that $\sum\limits_{j=1}^ma_j\gamma_j =\sum\limits_{j=2}^ma_{j}d_{j-1}/d_j-\sum\limits_{j=1}^ma_j+a_k$, $\alpha_{i_1,\dots,i_m}\in\mathbb{C}$ and the sum of the right-hand side is over all $(i_1,\dots ,i_m)\in B(A_m)$ such that $\sum\limits_{j=1}^ma_ji_j<\sum\limits_{j=2}^ma_jd_{j-1}/d_j-\sum\limits_{j=1}^ma_j+a_k$. If $\alpha_{i_1,\dots,i_m}\neq0$, then $\alpha_{i_1,\dots,i_m}$ contains the coefficients of the defining equations.
\end{Lemma}

See Appendix for proof.

\begin{Lemma}\label{lemma6}We have
\begin{gather*}
\det H(P,Q)=\prod_{i=2}^m\big(x_i^{d_{i-1}/d_i-1}+x_i^{d_{i-1}/d_i-2}y_i+\cdots+x_iy_i^{d_{i-1}/d_i-2}+y_i^{d_{i-1}/d_i-1}\big) \\
\hphantom{\det H(P,Q)=}{} +\sum\beta_{i_1,\dots,i_m;j_1,\dots,j_m}x_1^{i_1}\cdots x_m^{i_m}y_1^{j_1}\cdots y_m^{j_m},
\end{gather*}
where $\beta_{i_1,\dots,i_m;j_1,\dots,j_m}\in\mathbb{C}$ and the sum of the right-hand side is over all $(i_1,\dots,i_m), (j_1,\dots ,j_m)\in B(A_m)$ such that $\sum\limits_{k=1}^ma_k(i_k+j_k)<\sum\limits_{k=2}^ma_k(d_{k-1}/d_k-1)$. If $\beta_{i_1,\dots,i_m;j_1,\dots,j_m}\neq0$, then $\beta_{i_1,\dots,i_m;j_1,\dots,j_m}$ contains the coefficients of the defining equations.
\end{Lemma}

\begin{proof} By the def\/inition of $\det H(P,Q)$, when we expand the determinant $\det H(P,Q)$, the terms which do not contain the coef\/f\/icients of the def\/ining equations are
\begin{gather*}
\prod_{i=2}^m\frac{x_i^{d_{i-1}/d_i}-y_i^{d_{i-1}/d_i}}{x_i-y_i}=\prod_{i=2}^m
\big(x_i^{d_{i-1}/d_i-1}+x_i^{d_{i-1}/d_i-2}y_i+\cdots+x_iy_i^{d_{i-1}/d_i-2}+y_i^{d_{i-1}/d_i-1}\big).
\end{gather*}
Therefore we obtain Lemma \ref{lemma6}.
\end{proof}

\begin{Lemma}\label{lemma7}Let
\begin{gather*}
\det G_1(Q)\det H(P,Q)=\sum\gamma_{i_1,\dots,i_m;j_1,\dots,j_m}x_1^{i_1}\cdots x_m^{i_m}y_1^{j_1}\cdots y_m^{j_m},
\end{gather*}
where $(i_1,\dots,i_m), (j_1,\dots ,j_m)\in B(A_m)$. For $i_1\ge1$, we have $\gamma_{i_1,k_2,\dots,k_m;k_1+\ell_1+2-i_1,\ell_2,\dots,\ell_m}=\gamma_{i_1,\ell_2,\dots,\ell_m;k_1+\ell_1+2-i_1,k_2,\dots,k_m}=0$. Furthermore we have
\begin{gather*}
\gamma_{0,k_2,\dots,k_m;k_1+\ell_1+2,\ell_2,\dots,\ell_m}=a_1.
\end{gather*}
\end{Lemma}

\begin{proof}
Note that $\det G_1(Q)\det H(P,Q)$ is homogeneous of degree $2\sum\limits_{i=2}^ma_i(d_{i-1}/d_i-1)$ and
$\deg x_1^{i_1}x_2^{k_2}\cdots x_m^{k_m}y_1^{k_1+\ell_1+2-i_1}y_2^{\ell_2}\cdots y_m^{\ell_m}=2\sum\limits_{i=2}^ma_i(d_{i-1}/d_i-1)$.

Therefore, if $\gamma_{i_1,k_2,\dots,k_m;k_1+\ell_1+2-i_1,\ell_2,\dots,\ell_m}\neq0$, then
$\gamma_{i_1,k_2,\dots,k_m;k_1+\ell_1+2-i_1,\ell_2,\dots,\ell_m}$ does not contain the coef\/f\/icients of the def\/ining equations.
From Lemma~\ref{lemma6}, we have $i_1=0$.

Similarly, if $\gamma_{i_1,\ell_2,\dots,\ell_m;k_1+\ell_1+2-i_1,k_2,\dots,k_m}\neq0$, then $i_1=0$. From Lemmas \ref{lemma5},~\ref{lemma6} and~(\ref{eq-2-5}), we obtain $\gamma_{0,k_2,\dots,k_m;k_1+\ell_1+2,\ell_2,\dots,\ell_m}=a_1$.
\end{proof}

\begin{Lemma}\label{lemma8}
Let
\begin{gather*}
\frac{\partial \det H}{\partial y_i}(P,Q)\det G_i(Q)=\sum\delta_{i_1,\dots,i_m;j_1,\dots,j_m}x_1^{i_1}\cdots x_m^{i_m}y_1^{j_1}\cdots y_m^{j_m},
\end{gather*}
where $(i_1,\dots,i_m), (j_1,\dots ,j_m)\in B(A_m)$.
For $i_1\ge1$, we have $\delta_{i_1,k_2,\dots,k_m;k_1+\ell_1+1-i_1,\ell_2,\dots,\ell_m}=\delta_{i_1,\ell_2,\dots,\ell_m;k_1+\ell_1+1-i_1,k_2,\dots,k_m}=0$.
Furthermore we have
\begin{gather}
\delta_{0,k_2,\dots,k_m;k_1+\ell_1+1,\ell_2,\dots,\ell_m}=0\qquad \mbox{if}\quad i=1\label{ppppp}
\end{gather}
and
\begin{gather}
\delta_{0,k_2,\dots,k_m;k_1+\ell_1+1,\ell_2,\dots,\ell_m}=\left(\frac{d_{i-1}}{d_i}-1-k_i\right)(-1)^{i+1}a_i \qquad \mbox{if}\quad 2\le i\le m. \label{1}
\end{gather}
\end{Lemma}

\begin{proof}
Note that $\frac{\partial \det H}{\partial y_i}(P,Q)\det G_i(Q)$ is homogeneous of degree $2\sum\limits_{j=2}^m(d_{j-1}/d_{j}-1)a_{j}-a_1$
and $\deg x_1^{i_1}x_2^{k_2}\cdots x_m^{k_m}y_1^{k_1+\ell_1+1-i_1}y_2^{\ell_2}\cdots y_m^{\ell_m}=2\sum\limits_{j=2}^m(d_{j-1}/d_{j}-1)a_{j}-a_1$.

Therefore, if $\delta_{i_1,k_2,\dots,k_m;k_1+\ell_1+1-i_1,\ell_2,\dots,\ell_m}\neq0$, then
$\delta_{i_1,k_2,\dots,k_m;k_1+\ell_1+1-i_1,\ell_2,\dots,\ell_m}$ does not contain the coef\/f\/icients of the def\/ining equations.
From Lemma~\ref{lemma6}, we have $i_1=0$.

Similarly, if $\gamma_{i_1,\ell_2,\dots,\ell_m;k_1+\ell_1+1-i_1,k_2,\dots,k_m}\neq0$, then $i_1=0$.
From Lemmas \ref{lemma5}, \ref{lemma6} and (\ref{eq-2-5}), we obtain (\ref{ppppp}) and (\ref{1}).
\end{proof}

Let
\begin{gather*}
\widehat{\omega}(P,Q)=\frac{F(P,Q)}{(x_1-y_1)^2\det G_1(P)\det G_1(Q)}dx_1dy_1.
\end{gather*}

\begin{Lemma}\label{10}
For $0\le i_1\le k_1$, we have
\begin{gather}
c_{i_1,k_2,\dots,k_m;k_1+\ell_1-i_1,\ell_2,\dots,\ell_m}=i_1c_{1,k_2,\dots,k_m;k_1+\ell_1-1,\ell_2,\dots,\ell_m}\nonumber\\
\hphantom{c_{i_1,k_2,\dots,k_m;k_1+\ell_1-i_1,\ell_2,\dots,\ell_m}=}{}
+(1-i_1)c_{0,k_2,\dots,k_m;k_1+\ell_1,\ell_2,\dots,\ell_m}. \label{23}
\end{gather}
\end{Lemma}

\begin{proof}
If $k_1=0,1$, then Lemma \ref{10} holds obviously. Assume $k_1\ge2$. For $2\le i_1\le k_1$, from Lemmas \ref{lemma7} and~\ref{lemma8}, the coef\/f\/icient of $x_1^{i_1}x_2^{k_2}\cdots x_m^{k_m}y_1^{k_1+\ell_1+2-i_1}y_2^{\ell_2}\cdots y_m^{\ell_m}$ in $F(P,Q)$ is
\begin{gather*}
c_{i_1,k_2,\dots,k_m;k_1+\ell_1-i_1,\ell_2,\dots,\ell_m}\!-2c_{i_1-1,k_2,\dots,k_m;k_1+\ell_1-i_1+1,\ell_2,\dots,\ell_m}\!
+c_{i_1-2,k_2,\dots,k_m;k_1+\ell_1-i_1+2,\ell_2,\dots,\ell_m}.\!
\end{gather*}
On the other hand, from $k_1+\ell_1+2-i_1>1$, $k_1+\ell_1+2-i_1>\ell_1+1$, and Lemmas~\ref{lemma7},~\ref{lemma8}, the coef\/f\/icient of $x_1^{k_1+\ell_1+2-i_1}x_2^{\ell_2}\cdots x_m^{\ell_m}y_1^{i_1}y_2^{k_2}\cdots y_m^{k_m}$ in $F(P,Q)$ is zero. Therefore, from $\widehat{\omega}(P,Q)=\widehat{\omega}(Q,P)$, we have
\begin{gather}
c_{i_1,k_2,\dots,k_m;k_1+\ell_1-i_1,\ell_2,\dots,\ell_m}=2c_{i_1-1,k_2,\dots,k_m;k_1+\ell_1-i_1+1,\ell_2,\dots,\ell_m}\nonumber\\
\hphantom{c_{i_1,k_2,\dots,k_m;k_1+\ell_1-i_1,\ell_2,\dots,\ell_m}=}{}
-c_{i_1-2,k_2,\dots,k_m;k_1+\ell_1-i_1+2,\ell_2,\dots,\ell_m}. \label{24}
\end{gather}
We prove the equation (\ref{23}) by induction of $i_1$. For $i_1=0,1$, the equation~(\ref{23}) holds obviously. Assume that the equation (\ref{23}) holds for $i_1=n,n+1,\;(0\le n\le k_1-2)$. From (\ref{24}) and the assumption of induction, we have
\begin{gather*}
c_{n+2,k_2,\dots,k_m;k_1+\ell_1-n-2,\ell_2,\dots,\ell_m}=2c_{n+1,k_2,\dots,k_m;k_1+\ell_1-n-1,\ell_2,\dots,\ell_m}-c_{n,k_2,\dots,k_m;k_1+\ell_1-n,\ell_2,\dots,\ell_m} \\
\qquad{}
=(2n+2)c_{1,k_2,\dots,k_m;k_1+\ell_1-1,\ell_2,\dots,\ell_m}-2nc_{0,k_2,\dots,k_m;k_1+\ell_1,\ell_2,\dots,\ell_m}\\
\qquad\quad{}
-nc_{1,k_2,\dots,k_m;k_1+\ell_1-1,\ell_2,\dots,\ell_m}+(n-1)c_{0,k_2,\dots,k_m;k_1+\ell_1,\ell_2,\dots,\ell_m}\\
\qquad{}
=(n+2)c_{1,k_2,\dots,k_m;k_1+\ell_1-1,\ell_2,\dots,\ell_m}+(-n-1)c_{0,k_2,\dots,k_m;k_1+\ell_1,\ell_2,\dots,\ell_m}.
\end{gather*}
Therefore the equation (\ref{23}) holds for $i_1=n+2$.
\end{proof}

\begin{proof}[Proof of Proposition \ref{111}]
From Lemmas \ref{lemma7} and~\ref{lemma8}, the coef\/f\/icient of $x_2^{k_2}\cdots x_m^{k_m}y_1^{k_1+\ell_1+2}y_2^{\ell_2}$ $\cdots y_m^{\ell_m}$ in $F(P,Q)$ is
\begin{gather*}
c_{0,k_2,\dots,k_m;k_1+\ell_1,\ell_2,\dots,\ell_m}+a_1-\sum_{i=2}^m\left(\frac{d_{i-1}}{d_i}-1-k_i\right)a_i.
\end{gather*}
From $k_1+\ell_1+2> \ell_1+1$, $k_1+\ell_1+2>1$, and Lemmas~\ref{lemma7},~\ref{lemma8}, the coef\/f\/icient of $x_1^{k_1+\ell_1+2}x_2^{\ell_2}\cdots x_m^{\ell_m}$ $y_2^{k_2}\cdots y_m^{k_m}$ in $F(P,Q)$ is zero. Therefore, from $\widehat{\omega}(P,Q)=\widehat{\omega}(Q,P)$, we have
\begin{gather}
c_{0,k_2,\dots,k_m;k_1+\ell_1,\ell_2,\dots,\ell_m}+a_1-\sum_{i=2}^m\left(\frac{d_{i-1}}{d_i}-1-k_i\right)a_i=0. \label{12}
\end{gather}
If $k_1=0$, then from (\ref{eq-2-6})
\begin{gather*}
c_{0,k_2,\dots,k_m;\ell_1,\ell_2,\dots,\ell_m}=\sum_{i=2}^m\left(\frac{d_{i-1}}{d_i}-1\right)a_i-\sum_{i=1}^ma_ik_i-a_1 \\
\hphantom{c_{0,k_2,\dots,k_m;\ell_1,\ell_2,\dots,\ell_m}}{} =\sum_{i=2}^m\left(\frac{d_{i-1}}{d_i}-1\right)a_i-(2g-2)-a_1=1.
\end{gather*}
Therefore, if $k_1=0$, then Proposition~\ref{111} holds.

Assume $k_1\ge1$. From Lemmas \ref{lemma7} and \ref{lemma8}, the coef\/f\/icient of $x_1x_2^{k_2}\cdots x_m^{k_m}y_1^{k_1+\ell_1+1}y_2^{\ell_2}\cdots y_m^{\ell_m}$ in $F(P,Q)$ is
\begin{gather*}
c_{1,k_2,\dots,k_m;k_1+\ell_1-1,\ell_2,\dots,\ell_m}-2c_{0,k_2,\dots,k_m;k_1+\ell_1,\ell_2,\dots,\ell_m}+\sum_{i=2}^m\left(\frac{d_{i-1}}{d_i}-1-k_i\right)a_i.
\end{gather*}
Since $k_1\ge1$, we have $k_1+\ell_1+1>\ell_1+1$ and $k_1+\ell_1+1>1$. Therefore, from Lemmas \ref{lemma7} and~\ref{lemma8}, the coef\/f\/icient of $x_1^{k_1+\ell_1+1}x_2^{\ell_2}\cdots x_m^{\ell_m}y_1y_2^{k_2}\cdots y_m^{k_m}$ in $F(P,Q)$ is zero. Therefore, from $\widehat{\omega}(P,Q)=\widehat{\omega}(Q,P)$, we have
\begin{gather}
c_{1,k_2,\dots,k_m;k_1+\ell_1-1,\ell_2,\dots,\ell_m}-2c_{0,k_2,\dots,k_m;k_1+\ell_1,\ell_2,\dots,\ell_m}
+\sum_{i=2}^m\left(\frac{d_{i-1}}{d_i}-1-k_i\right)a_i=0.\label{13}
\end{gather}

From Lemma~\ref{10} and equations~(\ref{12}), (\ref{13}), (\ref{eq-2-6}), we have $c_{k_1,\dots,k_m;\ell_1,\dots,\ell_m}=1$.
\end{proof}

\begin{Proposition}\label{dr1}
We can take
\begin{gather*}
dr_1(Q)=-\frac{\varphi_{g+1}(Q)}{\det G_1(Q)}dy_1
\end{gather*}
such that $\widehat{\omega}(P,Q)=\widehat{\omega}(Q,P)$.
\end{Proposition}

\begin{proof}
From (\ref{1235}) and Proposition \ref{111}, $dr_1(Q)$ has the following form
\begin{gather*}
dr_1(Q)=-\frac{\varphi_{g+1}(Q)}{\det G_1(Q)}dy_1-\sum_{i=1}^gc_i'du_i(Q),
\end{gather*}
for certain constants $c_i'\in\mathbb{C}$. Let
\begin{gather*}\omega_1(P,Q)=c_1'du_1(P)du_1(Q)+\sum_{i=2}^gc_i'du_1(P)du_i(Q)+\sum_{i=2}^gc_i'du_i(P)du_1(Q).
\end{gather*}
Then $\omega_1(P,Q)$ is holomorphic and $\omega_1(P,Q)=\omega_1(Q,P)$. By adding $\omega_1(P,Q)$ to $\widehat{\omega}(P,Q)$, we obtain Proposition~\ref{dr1}.
\end{proof}

\begin{Remark}
There is a certain freedom of choice of the second kind dif\/ferentials $\{dr_i\}_{i=1}^g$, i.e., we can add a linear combination of the holomorphic one forms $\{du_i\}_{i=1}^g$ to $\{dr_i\}_{i=1}^g$. In~\cite{Eilers}, for hyperelliptic curves, it is discussed what choice of $\{dr_i\}_{i=1}^g$ is better for the problem that one considers.
\end{Remark}

\section{Sigma function of telescopic curves}\label{section4}

Let $X$ be a telescopic curve of genus $g\geq 1$ associated with $(a_1,\dots,a_m)$. We take a fundamental dif\/ferential of second kind
\begin{gather}
\widehat{\omega}(P,Q)=d_Q\Omega(P,Q)+\sum_{i=1}^g du_i(P)dr_i(Q),\label{alge}
\end{gather}
such that
\begin{gather}
dr_1(Q)=-\frac{\varphi_{g+1}(Q)}{\det G_1(Q)}dy_1.\label{spe}
\end{gather}
This choice is possible from Proposition~\ref{dr1}. The set $\{du_i,dr_i\}_{i=1}^g$ becomes a symplectic basis of the cohomology group $H^1(X,{\mathbb C})$ (see \cite{Aya1, N1}).

Take a symplectic basis $\{\alpha_i,\beta_i\}_{i=1}^g$ of the homology group and def\/ine the period matrices by
\begin{gather*}
2\omega_1=\left(\int_{\alpha_j}du_{i}\right), \qquad\! 2\omega_2=\left(\int_{\beta_j}du_{i}\right),\qquad\!
-2\eta_1=\left(\int_{\alpha_j}dr_i\right), \qquad\! -2\eta_2=\left(\int_{\beta_j}dr_i\right).
\end{gather*}
The normalized period matrix is given by $\tau=\omega_1^{-1}\omega_2$.

Let $\delta=\tau\delta'+\delta''$, $\delta',\delta''\in {\mathbb R}^g$ be the Riemann's constant with respect to the choice $(\{\alpha_i,\beta_i\},\infty)$. We set $\delta={}^t({}^t\delta', {}^t\delta'')$.

The sigma function $\sigma(u)$, $u=(u_1,\dots ,u_g)$ is def\/ined by
\begin{gather*}
\sigma(u)=C\exp\left(\frac{1}{2}{}^tu\eta_1\omega_1^{-1}u\right) \theta[\delta]\left((2\omega_1)^{-1}u,\tau\right),
\end{gather*}
where $\theta[\delta](u)$ is the Riemann's theta function with the characteristic $\delta$ def\/ined by
\begin{gather*}
\theta[\delta](u)=\sum_{n\in\mathbb{Z}^g}\exp\big\{\pi\sqrt{-1}\;{}^t(n+\delta')\tau(n+\delta')+2\pi\sqrt{-1}\;{}^t(n+\delta')(u+\delta'')\big\},
\end{gather*}
and $C$ is a constant. Since $\delta$ is a half-period from Lemma~\ref{lem-2-2}, $\sigma(u)$ vanishes on the Abel--Jacobi image of the $(g-1)$-th symmetric products of the telescopic curves. This property is important in the proof of Proposition~\ref{maineq}.

We have the following propositions.

\begin{Proposition}[\cite{Aya1, N1}]\label{period}
For $m_1,m_2\in\mathbb{Z}^g$, we have
\begin{gather*}
\frac{\sigma(u+2\omega_1m_1+2\omega_2m_2)}{\sigma(u)}
=(-1)^{2({}^t\delta'm_1-{}^t\delta''m_2)+{}^tm_1m_2}\\
\hphantom{\frac{\sigma(u+2\omega_1m_1+2\omega_2m_2)}{\sigma(u)}=}{}
\times\exp\big\{{}^t(2\eta_1m_1+2\eta_2m_2)(u+\omega_1m_1+\omega_2m_2)\big\}.
\end{gather*}
\end{Proposition}

\begin{Proposition}[\cite{Aya1, N1}]\label{odd}
We have $\sigma(-u)=(-1)^{|\mu(A_m)|}\sigma(u)$.
\end{Proposition}

The fundamental dif\/ferential of second kind $\widehat{\omega}(P,Q)$ is expressed by the sigma function as follows.

\begin{Proposition}\label{maineq}
If $\sum\limits_{i=1}^gP_i\in S^g(X\backslash\infty)$ is a general divisor, then for any $1\le i\le g$ we have
\begin{gather}
\widehat{\omega}(P,Q)=d_Pd_Q\log\sigma\left(\int_Q^Pdu-\sum_{j\neq i}\int_{\infty}^{P_j}du\right),\label{main}
\end{gather}
where $du={}^t(du_1,\dots,du_g)$.
\end{Proposition}

\begin{proof}
For simplicity we prove for $i=1$. Let $e=-\sum\limits_{j=2}^{g}\int_{\infty}^{P_j}du$. Then, from Proposition \ref{odd}, we have $\sigma(e)=0$. Let
\begin{gather*}
E(P,Q)=\sigma\left(\int_Q^Pdu-\sum_{j=2}^g\int_{\infty}^{P_j}du\right).
\end{gather*}
Suppose $E(P,Q)$ vanishes identically with respect to $P$, $Q$. Then we have $E(\infty,P_1)=0$. Therefore there exist $g-1$ points $P_1',\dots,P_{g-1}'\in X$ such that
\begin{gather*}
\sum_{i=1}^g\int_{\infty}^{P_i}du=\sum_{i=1}^{g-1}\int_{\infty}^{P_i'}du.
\end{gather*}
This contradicts the fact that $\sum\limits_{j=1}^gP_j\in S^g(X\backslash\infty)$ is a general divisor. Consequently, $E(P,Q)$ does not vanish identically with respect to $P$, $Q$. Therefore there exist $2g-2$ points $Q_1,\dots,Q_{g-1}$, $R_1,\dots,R_{g-1}\in X$ such that the divisor of zeros of $E(P,Q)$ is the sum of $\{(R,R)\,|\,R\in X\}$, $\{Q_j\}\times X$, $X\times\{R_j\}$ ($j=1,\dots,g-1$), including multiplicities (cf.~\cite[p.~156]{Mumford}). Let $\widetilde{\omega}(P,Q)$ be the right-hand side of~(\ref{main}). First we consider the series expansion of $\widetilde{\omega}(P,Q)$ around a~point~$(R,R)$. Let $t$ be a local coordinate around $R\in X$ such that $t(R)=0$ and $t_P$, $t_Q$ two copies of~$t$. Then we have the expansion $E(P,Q)=(t_P-t_Q)t_P^kt_Q^{\ell}f(t_P,t_Q)$ around $(R,R)$, where $k$, $\ell$ are nonnegative integers and $f(t_P,t_Q)$ is a holomorphic function of~$t_P$,~$t_Q$ satisfying $f(t_P,t_Q)\neq0$ for any~$t_P$,~$t_Q$. Hence, around $(R,R)$, we have the expansion
\begin{gather*}
\widetilde{\omega}(P,Q)=\frac{1}{(t_P-t_Q)^2}+(\text{holomorphic function of}\ t_P,t_Q).
\end{gather*}
Next we prove $\widetilde{\omega}(P,Q)$ is holomorphic around a point $(S_1,S_2)$ satisfying $S_1\neq S_2$. For a local coordinate $t_i$ around $S_i$ such that $t_i(S_i)=0$, $i=1,2$, we have the expansion $E(P,Q)=t_1^at_2^bg(t_1,t_2)$, where $a$, $b$ are nonnegative integers and $g(t_1,t_2)$ is a holomorphic function of $t_1$, $t_2$ satisfying $g(t_1,t_2)\neq0$ for any $t_1$, $t_2$. Hence $\widetilde{\omega}(P,Q)$ is holomorphic around $(S_1,S_2)$ satisfying $S_1\neq S_2$. Therefore $\widehat{\omega}(P,Q)-\widetilde{\omega}(P,Q)$ is holomorphic on $X\times X$. Consequently there exist constants $\{c_{ij}\}$ such that
\begin{gather}
\widehat{\omega}(P,Q)-\widetilde{\omega}(P,Q)=\sum_{ij}c_{ij}du_i(P)du_j(Q).\label{holcon}
\end{gather}
From (\ref{alge}) we have
\begin{gather*}
\int_{\alpha_j}\widehat{\omega}=-\, {}^tdu(P)(2\eta_1e_j),
\end{gather*}
where the integration is with respect to the second variable and $e_j$ is the $j$-th unit vector. On the other hand we have
\begin{gather*}
\int_{\alpha_j}\widetilde{\omega}=d_P\log\sigma\left(\int_{P_0}^Pdu-\sum_{j=2}^g\int_{\infty}^{P_j}du-2\omega_1e_j\right)-d_P\log\sigma\left(\int_{P_0}^Pdu-\sum_{j=2}^g\int_{\infty}^{P_j}du\right),
\end{gather*}
where $P_0$ is a base point of $\alpha_j$. From Proposition \ref{period} we have
\begin{gather*}
\int_{\alpha_j}\widetilde{\omega}=d_P\left\{-\,{}^t(2\eta_1e_j)\int_{P_0}^Pdu\right\}=-\, {}^t(2\eta_1e_j)du(P).
\end{gather*}
Therefore we have $\int_{\alpha_j}(\widehat{\omega}-\widetilde{\omega})=0$. If we set $C=(c_{ij})$, then from~(\ref{holcon}) we have ${}^tdu(P)\cdot C\cdot(2\omega_1e_j)=0$. Hence we have $C\cdot(2\omega_1e_j)=0$ for any $j$, i.e., $C\omega_1=0$. Since $\omega_1$ is a regular matrix, we have $C=0$. Therefore we have $\widehat{\omega}=\widetilde{\omega}$.
\end{proof}

We def\/ine the function
\begin{gather*}
\wp_{i,j}(u)=-\frac{\partial^2}{\partial u_i\partial u_j}\log\sigma(u).
\end{gather*}
Then we have
\begin{gather}
\wp_{i,j}(u)=\frac{\sigma_i(u)\sigma_j(u)-\sigma_{i,j}(u)\sigma(u)}{\sigma(u)^2},\label{178}
\end{gather}
where $\sigma_i(u)=\frac{\partial}{\partial u_i}\sigma(u)$ and $\sigma_{i,j}(u)=\frac{\partial^2}{\partial u_i\partial u_j}\sigma(u)$.

We have
\begin{gather*}
d_Pd_Q\log\sigma\left(\int_Q^Pdu-\sum_{j\neq i}\int_{\infty}^{P_j}du\right) \\
\qquad {}=\sum_{k,\ell=1}^g\wp_{k,\ell}\left(\int_Q^Pdu-\sum_{j\neq i}\int_{\infty}^{P_j}du\right)\frac{\varphi_{g+1-k}(P)\varphi_{g+1-\ell}(Q)}{\det G_1(P)\det G_1(Q)}dx_1dy_1.
\end{gather*}

From Proposition \ref{maineq}, we have
\begin{gather}
\frac{F(P,Q)}{(x_1-y_1)^2}=\sum_{k,\ell=1}^g\wp_{k,\ell}\left(\int_Q^Pdu-\sum_{j\neq i}\int_{\infty}^{P_j}du\right)\varphi_{g+1-k}(P)\varphi_{g+1-\ell}(Q),\label{maineq2}
\end{gather}
as a meromorphic function of $(P,Q)\in X^2$. This formula is an analogue of the formula of Klein (cf.~\cite[Theorem 3.4]{EEL}).

\section{Frobenius--Stickelberger matrix}\label{section5}

For $P_1,\dots,P_k,P\in X$, we def\/ine the matrix (Frobenius--Stickelberger matrix) as in the case of~\cite{Matsutani}
\begin{gather}
\left(\begin{matrix}
\varphi_1(P_1) & \varphi_2(P_1) & \cdots & \varphi_{k+1}(P_1) \\
\varphi_1(P_2) & \varphi_2(P_2) & \cdots & \varphi_{k+1}(P_2) \\
\vdots & \vdots &\ddots & \vdots \\
\varphi_1(P_k) & \varphi_2(P_k) & \cdots & \varphi_{k+1}(P_k) \\
\varphi_1(P) & \varphi_2(P) & \cdots & \varphi_{k+1}(P)
\end{matrix}\right).\label{eq-2-9}
\end{gather}
For $1\le i\le k+1$, let $\psi_{k+1}(P_1,\dots,P_k;P)$ and $\psi_k^{(i)}(P_1,\dots,P_k)$ be the determinants of the matrix (\ref{eq-2-9}) and the matrix obtained by deleting the last row and the $i$-th column from (\ref{eq-2-9}), respectively. Note that $\psi_k^{(i)}(P_1,\dots,P_k)$ does not vanish identically as a meromorphic function of $P_1,\dots,P_k$. We def\/ine $\mu_{k+1}(P_1,\dots,P_k;P)$ by
\begin{gather*}
\mu_{k+1}(P_1,\dots,P_k;P)=\frac{\psi_{k+1}(P_1,\dots,P_k;P)}{\psi_k^{(k+1)}(P_1,\dots,P_k)}
\end{gather*}
and $\mu_{k,i}(P_1,\dots,P_k)$ by
\begin{gather*}
\mu_{k+1}(P_1,\dots,P_k;P)=\sum_{i=1}^{k+1}(-1)^{k+1-i}\mu_{k,i}(P_1,\dots,P_k)\varphi_i(P)
\end{gather*}
and $\mu_{k,i}(P_1,\dots,P_k)=0$ for $i\ge k+2$. Then, for $1\le i\le k+1$, we have
\begin{gather*}
\mu_{k,i}(P_1,\dots,P_k)=\frac{\psi_k^{(i)}(P_1,\dots,P_k)}{\psi_k^{(k+1)}(P_1,\dots,P_k)}.
\end{gather*}
Note that $\mu_{k,i}(P_1,\dots,P_k)$ can be regarded as a meromorphic function on $X^{k}$.

\begin{Proposition}\label{seq}
Let $z_k$ be the local parameter of $P_k$ around $\infty$ satisfying~\eqref{eq-3-1}. Then, as a~meromorphic function of $P_1,\dots,P_k$, we have the expansion
\begin{gather*}
\mu_{k,i}(P_1,\dots,P_k)=\mu_{k-1,i}(P_1,\dots,P_{k-1})z_k^{N(k)-N(k+1)}+O\big(z_k^{N(k)-N(k+1)+1}\big),
\end{gather*}
where $1\le i\le k$ and $N(n)=\operatorname{ord}_{\infty}(\varphi_n)$ for a positive integer~$n$.
\end{Proposition}

\begin{proof}
As a meromorphic function of $P_1,\dots,P_k$, we have
\begin{gather*}
\mu_{k,i}(P_1,\dots,P_k)=\frac{\psi_k^{(i)}(P_1,\dots,P_k)}{\psi_k^{(k+1)}(P_1,\dots,P_k)}
=\frac{\varphi_{k+1}(P_k)\psi_{k-1}^{(i)}(P_1,\dots,P_{k-1})+O\big(z_k^{-N(k+1)+1}\big)}
{\varphi_k(P_k)\psi_{k-1}^{(k)}(P_1,\dots,P_{k-1})+O\big(z_k^{-N(k)+1}\big)} \\
\hphantom{\mu_{k,i}(P_1,\dots,P_k)}{}
=\frac{\big\{z_k^{-N(k+1)}+O\big(z_k^{-N(k+1)+1}\big)\big\}\psi_{k-1}^{(i)}(P_1,\dots,P_{k-1})+O\big(z_k^{-N(k+1)+1}\big)}{\big\{z_k^{-N(k)}
+O\big(z_k^{-N(k)+1}\big)\big\}\psi_{k-1}^{(k)}(P_1,\dots,P_{k-1})+O\big(z_k^{-N(k)+1}\big)}\\
\hphantom{\mu_{k,i}(P_1,\dots,P_k)}{}
=\frac{z_k^{-N(k+1)}}{z_k^{-N(k)}}\cdot\frac{\psi_{k-1}^{(i)}(P_1,\dots,P_{k-1})+O(z_k)}{\psi_{k-1}^{(k)}(P_1,\dots,P_{k-1})+O(z_k)}\\
\hphantom{\mu_{k,i}(P_1,\dots,P_k)}{}
=\mu_{k-1,i}(P_1,\dots,P_{k-1})z_k^{N(k)-N(k+1)}+O\big(z_k^{N(k)-N(k+1)+1}\big). \tag*{\qed}
\end{gather*}
\renewcommand{\qed}{}
\end{proof}

\section{Riemann's singularity theorem}\label{section6}

Let $X$ be a telescopic curve of genus $g\ge1$. For a divisor $D$, let $L(D)$ be the vector space consisting of meromorphic functions $f$ on $X$ such that $\operatorname{div}(f)+D\ge0$ and the zero function on~$X$, and $\ell(D)$ the dimension of $L(D)$.

For $1\le k\le g-1$ and $P_1,\dots,P_k\in X\backslash\infty$, let
\begin{gather*}
u^{[k]}=\sum_{i=1}^k\int_{\infty}^{P_k}du
\end{gather*}
and
\begin{gather*}
n_k=\ell(P_1+\cdots+P_k+(g-k-1)\infty).
\end{gather*}
Then the following theorem holds.

\begin{Theorem}[Riemann's singularity theorem, cf.~\cite{ACGH,Matsutani2, Mumford}]\label{riemann}\quad
\begin{enumerate}\itemsep=0pt
\item[$1.$] For every multi-index $(\alpha_1,\dots,\alpha_m)$ with $\alpha_i\in\{1,\dots,g\}$ and $m<n_k$,
\begin{gather*}
\frac{\partial^m}{\partial u_{\alpha_1}\cdots \partial u_{\alpha_m}}\sigma\big(u^{[k]}\big)=0.
\end{gather*}

\item[$2.$] There exists a multi-index $(\beta_1,\dots,\beta_{n_k})$, which in general depends on $P_1+\cdots+P_k$, such that
\begin{gather*}
\frac{\partial^{n_k}}{\partial u_{\beta_1}\cdots \partial u_{\beta_{n_k}}}\sigma\big(u^{[k]}\big)\neq0.
\end{gather*}
\end{enumerate}
\end{Theorem}

The following proposition is stated for the hyperelliptic curves in \cite{O} and for the curves $y^r=f(x)$ in \cite{Matsutani,Matsutani2}. The same statement is also satisf\/ied for telescopic curves. The proof is similar to \cite[Proposition~5.2]{O}.

\begin{Proposition}\label{dim} If $\psi_k^{(k+1)}(P_1,\dots,P_k)\neq0$, then we have $n_k=\sharp\{n\,|\,0\le N(n)\le g-k-1\}$, where $N(n)=\operatorname{ord}_{\infty}(\varphi_n)$ for a positive integer $n$ and $\sharp$ means the number of elements.
\end{Proposition}

\section{Jacobi inversion formulae for telescopic curves}\label{section7}

For $1\le k\le g$ and $P_1,\dots,P_k\in X\backslash\infty$, let
\begin{gather*}
u^{[k]}=\sum_{j=1}^k\int_{\infty}^{P_j}du.
\end{gather*}

\subsection[$k=g$]{$\boldsymbol{k=g}$}

\begin{Theorem}\label{main2} As a meromorphic function of $P_1,\dots,P_g$, we have
\begin{gather}
\wp_{1,i}\big(u^{[g]}\big)=(-1)^{i-1}\mu_{g,g+1-i}(P_1,\dots,P_g),\qquad 1\le i\le g.\label{rem}
\end{gather}
\end{Theorem}

\begin{proof}Let $S$ be the set of $(P_1,\dots,P_g)\in(X\backslash\infty)^g$ such that $\sum\limits_{i=1}^gP_i$ is a general divisor and $P_i\neq P_j$ for any $i$, $j$ $(i\neq j)$. First we prove the equation (\ref{rem}) for $(P_1,\dots,P_g)\in S$. From Lemma~\ref{167}(ii), we have $\operatorname{ord}_{\infty}(\det H(P,Q))\le\sum\limits_{i=2}^m(d_{i-1}/d_i-1)a_i$ and $\operatorname{ord}_{\infty}(x_1\frac{\partial \det H}{\partial y_k}(P,Q))\le\sum\limits_{i=2}^m(d_{i-1}/d_i-1)a_i-a_k+a_1$ with respect to $P$. On the other hand, we have $\operatorname{ord}_{\infty}(x_1^2\varphi_g(P))=-1+a_1+\sum\limits_{i=2}^m(d_{i-1}/d_i-1)a_i$. Since $a_i\ge2$ for any $i$, we have $\operatorname{ord}_{\infty}(\det H(P,Q))<\operatorname{ord}_{\infty}(x_1^2\varphi_g(P))$ and $\operatorname{ord}_{\infty}(x_1\frac{\partial \det H}{\partial y_k}(P,Q))<\operatorname{ord}_{\infty}(x_1^2\varphi_g(P))$ with respect to~$P$. We let $P\to\infty$ after dividing the both sides of~(\ref{maineq2}) by $\varphi_{g}(P)$ and $Q=P_i$. Then, from~(\ref{spe}), we obtain
\begin{gather*}
\varphi_{g+1}(P_i)=\sum_{\ell=1}^g\wp_{1,\ell}\big(u^{[g]}\big)\varphi_{g+1-\ell}(P_i)=\sum_{j=1}^g\wp_{1, g+1-j}\big(u^{[g]}\big)\varphi_j(P_i),
\end{gather*}
where we use the fact that $\wp_{1,\ell}(u)$ is an even function from Proposition~\ref{odd}. From $(P_1,\dots,P_g)\in S$, we have $\psi_g^{(g+1)}(P_1,\dots,P_g)\neq0$ (cf.\ \cite[p.~154]{ACGH}). From $\mu_{g+1}(P_1,\dots,P_g;P_i)=0$ for any $i$, we have
\begin{gather*}
\mu_{g+1}(P_1,\dots,P_g;P_i)=\varphi_{g+1}(P_i)+\sum_{j=1}^{g}(-1)^{g+1-j}\mu_{g,j}(P_1,\dots,P_g)\varphi_j(P_i)=0.
\end{gather*}
Therefore we have
\begin{gather*}
\sum_{j=1}^g\wp_{1,g+1-j}\big(u^{[g]}\big)\varphi_j(P_i)=\sum_{j=1}^g(-1)^{g-j}\mu_{g,j}(P_1,\dots,P_g)\varphi_j(P_i).
\end{gather*}
Therefore we have
\begin{gather*}
A\cdot
\left(\begin{matrix}
\alpha_1 \\
\vdots \\
\alpha_g
\end{matrix}\right)
=\left(\begin{matrix}
0 \\
\vdots \\
0
\end{matrix}\right),
\end{gather*}
where $A=(\varphi_j(P_i))_{1\le i,j\le g}$ and $\alpha_j=\wp_{1,g+1-j}\big(u^{[g]}\big)-(-1)^{g-j}\mu_{g,j}(P_1,\dots,P_g)$. From $(P_1,\dots,P_g)$ $\in S$, we have $\det A\neq0$. Therefore we have $\alpha_j=0$, i.e., $\wp_{1,g+1-j}\big(u^{[g]}\big)=(-1)^{g-j}\mu_{g,j}(P_1,\dots,P_g)$ for any $j$. We set $i=g+1-j$, then we have $\wp_{1,i}(u^{[g]})=(-1)^{i-1}\mu_{g,g+1-i}(P_1,\dots,P_g)$.

Let $T$ be the set of $(P_1,\dots,P_g)\in(X\backslash\infty)^g$ such that $\psi_g^{(g+1)}(P_1,\dots,P_g)\neq0$. Then we have $S=T$ (cf.~\cite[p.~154]{ACGH}).
Since the equation~(\ref{rem}) holds for any $(P_1,\dots,P_g)\in T$, it holds as a~meromorphic function of $P_1,\dots,P_g$.
\end{proof}

\begin{Remark}
As discussed in \cite{Matsutani}, for a hyperelliptic curve, $\mu_{g,g+1-i}$ is equal to the symmetric polynomial $e_i$.
Therefore Theorem~\ref{main2} is a~natural generalization of the Jacobi inversion formulae for hyperelliptic curves to telescopic curves.
\end{Remark}

\subsection[$k\le g-1$]{$\boldsymbol{k\le g-1}$}

Let $a=\min\{a_1,\dots,a_m\}$. Hereafter we assume $g-a\le k\le g-1$.

\begin{Theorem}\label{main3}
$\sigma_{g-k}\big(u^{[k]}\big)$ does not vanish identically with respect to $P_1,\dots,P_k$ and we have, as a~meromorphic function of $P_1,\dots,P_k$,
\begin{gather*}
\frac{\sigma_i\big(u^{[k]}\big)}{\sigma_{g-k}\big(u^{[k]}\big)}=(-1)^{k+i-g}\mu_{k,g+1-i}(P_1,\dots,P_{k}), \qquad 1\le i\le g.
\end{gather*}
\end{Theorem}

\begin{proof}
First we prove for $k=g-1$. Let $z_g$ be the local parameter of $P_g$ around $\infty$ satis\-fying~(\ref{eq-3-1}). From~\cite[Theorem~2]{Aya2}, $\sigma_1\big(u^{[g-1]}\big)$ does not vanish identically with respect to $P_1,\dots$, $P_{g-1}$ and we have the expansion
\begin{gather*}
\sigma\left(u^{[g-1]}+\int_{\infty}^{P_g}du\right) =\sigma_1\big(u^{[g-1]}\big)z_g+O\big(z_g^2\big).
\end{gather*}
From Theorem \ref{main2} and (\ref{178}), we have
\begin{gather*}
\frac{\sigma_i\left(u^{[g-1]}+ {\int_{\infty}^{P_g}du}\right)\sigma_1\left(u^{[g-1]}+ {\int_{\infty}^{P_g}du}\right)-\sigma_{1,i}\left(u^{[g-1]}+ {\int_{\infty}^{P_g}du}\right)\sigma\left(u^{[g-1]}+ {\int_{\infty}^{P_g}du}\right)}
{\sigma\left(u^{[g-1]}+{\int_{\infty}^{P_g}du}\right)^2} \\
\qquad {} =(-1)^{i-1}\mu_{g,g+1-i}(P_1,\dots,P_g),
\end{gather*}
as a meromorphic function of $P_1,\dots,P_g$. Therefore, from Proposition~\ref{seq}, we have
\begin{gather*}
\frac{\big\{\sigma_i\big(u^{[g-1]}\big)+O(z_g)\big\}\cdot\big\{\sigma_1\big(u^{[g-1]}\big)+O(z_g)\big\}+O(z_g)}{\sigma_1\big(u^{[g-1]}\big)^2z_g^2+O(z_g^3)} \\
\qquad{} =(-1)^{i-1}z_g^{-2}\mu_{g-1,g+1-i}(P_1,\dots,P_{g-1})+O\big(z_g^{-1}\big).
\end{gather*}
By comparing the coef\/f\/icient of $z_g^{-2}$ of the above equation, we f\/ind that $\sigma_i\big(u^{[g-1]}\big)$ does not vanish identically with respect to $P_1,\dots,P_{g-1}$ and we have, as a meromorphic function of $P_1,\dots,P_{g-1}$,
\begin{gather*}
\frac{\sigma_i\big(u^{[g-1]}\big)}{\sigma_1\big(u^{[g-1]}\big)}=(-1)^{i-1}\mu_{g-1,g+1-i}(P_1,\dots,P_{g-1}). 
\end{gather*}

Next we prove Theorem \ref{main3} for $g-a\le k\le g-2$ by induction of $k$ as in the case of~\cite{Matsutani}. Assume that Theorem~\ref{main3} holds for~$k$ satisfying $g-a+1\le k\le g-1$. Then, for $i\ge g-k$, $\sigma_i\big(u^{[k]}\big)$ does not vanish identically with respect to $P_1,\dots,P_k$. From the assumption of induction, we have, as a meromorphic function of $P_1,\dots,P_k$,
\begin{gather*}
\frac{\sigma_{g-k+1}\big(u^{[k]}\big)}{\sigma_{g-k}\big(u^{[k]}\big)}=-\mu_{k,k}(P_1,\dots,P_k).
\end{gather*}
From Proposition \ref{seq}, we have
\begin{gather*}
\mu_{k,k}(P_1,\dots,P_k)=z_k^{N(k)-N(k+1)}+O\big(z_k^{N(k)-N(k+1)+1}\big).
\end{gather*}
By the assumption of induction we have, as a meromorphic function of $P_1,\dots,P_k$,
\begin{gather}
\frac{\sigma_i\big(u^{[k]}\big)}{\sigma_{g-k}\big(u^{[k]}\big)}=(-1)^{k+i-g}\mu_{k,g+1-i}(P_1,\dots,P_k).\label{ind}
\end{gather}
By multiplying the both sides of~(\ref{ind}) by $\sigma_{g-k}\big(u^{[k]}\big)/\sigma_{g-k+1}\big(u^{[k]}\big)$, we have, as a meromorphic function of $P_1,\dots,P_k$,
\begin{gather*}
\frac{\sigma_{g-k}\big(u^{[k]}\big)}{\sigma_{g-k+1}\big(u^{[k]}\big)}\cdot\frac{\sigma_i\big(u^{[k]}\big)}{\sigma_{g-k}\big(u^{[k]}\big)}
=(-1)^{k+i-g}\frac{\sigma_{g-k}\big(u^{[k]}\big)}{\sigma_{g-k+1}\big(u^{[k]}\big)}\cdot\mu_{k,g+1-i}(P_1,\dots,P_k).
\end{gather*}
Therefore, from Proposition \ref{seq}, we have
\begin{gather}
\frac{\sigma_i\big(u^{[k]}\big)}{\sigma_{g-k+1}\big(u^{[k]}\big)}
=(-1)^{k-1+i-g}\big\{z_k^{N(k+1)-N(k)}+O\big(z_k^{N(k+1)-N(k)+1}\big)\big\}\cdot\mu_{k,g+1-i}(P_1,\dots,P_k)\nonumber \\
\hphantom{\frac{\sigma_i\big(u^{[k]}\big)}{\sigma_{g-k+1}\big(u^{[k]}\big)}}{} =(-1)^{k-1+i-g}\mu_{k-1,g+1-i}(P_1,\dots,P_{k-1})+O(z_k).\label{6}
\end{gather}
Since $\psi_{k-1}^{(k)}(P_1,\dots,P_{k-1})$ does not vanish identically with respect to $P_1,\dots,P_{k-1}$, there exist $\tilde{P}_1,\dots,\tilde{P}_{k-1}\in X\backslash\infty$ such that $\psi_{k-1}^{(k)}(\tilde{P}_1,\dots,\tilde{P}_{k-1})\neq0$. Let $\tilde{u}^{[k-1]}=\sum\limits_{i=1}^{k-1}\int_{\infty}^{\tilde{P}_i}du$. From $g-a<k$, we have $g-k<a$. Therefore, from Theorem~\ref{riemann} and Proposition~\ref{dim}, there exists~$i_0$ such that $\sigma_{i_0}(\tilde{u}^{[k-1]})\neq0$. Therefore $\sigma_{i_0}\big(u^{[k-1]}\big)$ does not vanish identically with respect to $P_1,\dots,P_{k-1}$. Since the equation (\ref{6}) holds for $i=i_0$, we f\/ind that $\sigma_{g-k+1}\big(u^{[k-1]}\big)$ does not vanish identically with respect to $P_1,\dots,P_{k-1}$. Take the limit $P_k\to\infty$ in (\ref{6}), then we have
\begin{gather*}
\frac{\sigma_i\big(u^{[k-1]}\big)}{\sigma_{g-k+1}\big(u^{[k-1]}\big)}=(-1)^{k-1+i-g}\mu_{k-1,g+1-i}(P_1,\dots,P_{k-1}).
\end{gather*}
Therefore Theorem \ref{main3} holds for $k-1$.
\end{proof}

\begin{Corollary}\label{999}
If $g=a-1,a,a+1$, then we have
\begin{gather}
\frac{\sigma_g\big(u^{[1]}\big)}{\sigma_{g-1}\big(u^{[1]}\big)}=-x_{i_0}(P_1),\label{xcoordinate}
\end{gather}
where $i_0$ is determined by $a_{i_0}=\arg\min\{a_1,\dots,a_m\}$.
\end{Corollary}

\begin{proof}
If $g=a-1,a,a+1$, then Theorem \ref{main3} holds for $k=1$.
\end{proof}

\begin{Remark}
Corollary \ref{999} asserts that the $x_{i_0}$ coordinate of $P_1$ is expressed by the sigma function. For example, Corollary~\ref{999} holds for $(4,6,5)$ curves.
\end{Remark}

\begin{Remark}
For $(2,5)$ and $(2,7)$ curves, it is known that $x_2$ coordinate of~$P_1$ can be expressed explicitly by the sigma function (see \cite[Lemma~3.2.4]{O2} and \cite[p.~221]{BEL3}). For example, for $(2,5)$ curves, it is known that
\begin{gather*}
x_2(P_1)=\frac{1}{2}\cdot\frac{\sigma\big(2u^{[1]}\big)}{\sigma_1\big(u^{[1]}\big)^4}.
\end{gather*}
On the other hand, for $(2,5)$ and $(2,7)$ curves, it is known that the expression of $x_2$ coordinate of $P_1$ can also be derived by dif\/ferentiating the both sides of (\ref{xcoordinate}) (see \cite[p.~221]{BEL3} and \cite[Remark~5.4]{Matsutani}). For example, for $(2,5)$ curves, it is known that
\begin{gather*}
x_2(P_1)=\frac{1}{2}\cdot\frac{\sigma_{11}\big(u^{[1]}\big)x_1(P_1)^2
+2\sigma_{12}\big(u^{[1]}\big)x_1(P_1)+\sigma_{22}\big(u^{[1]}\big)}{\sigma_1\big(u^{[1]}\big)}.
\end{gather*}

Although the similar expressions for the other coordinates of telescopic curves are not obtained currently, we will consider a generalization of these results to telescopic curves in a~subsequent work.
\end{Remark}

\begin{Remark}
Theorem \ref{main2} holds for $\widehat{\omega}(P,Q)$ satisfying (\ref{spe}). On the other hand, Theorem \ref{main3} holds for any choice of $\widehat{\omega}(P,Q)$.
\end{Remark}

\begin{Remark}
As mentioned in \cite{Matsutani2}, Theorem \ref{main3} for $k=g-1$ can also be proved by \cite[Theorem~1]{N1} and \cite[Theorem~1]{J}.
\end{Remark}

\begin{Remark} In this paper, we consider the Jacobi inversion formulae for the telescopic curves, which the Young diagrams are symmetric, i.e., the vector of Riemann constants for a~base point~$\infty$ is a half-period. On the other hand, in~\cite{Matsutani4, Matsutani3}, the Jacobi inversion formulae are derived for $(3,4,5)$ curves and $(3,7,8)$ curves, which the Young diagrams are not symmetric, i.e., the vector of Riemann constants for a base point $\infty$ is not a half-period.
\end{Remark}

\section[Example: $(4,6,5)$-curve]{Example: $\boldsymbol{(4,6,5)}$-curve}\label{section8}

In this section we give an explicit example of the Jacobi inversion formulae in the case of a~$(4,6,5)$-curve $X$. The genus of $X$ is 4 and $\varphi_1=1$, $\varphi_2=x_1$, $\varphi_3=x_3$, $\varphi_4=x_2$, $\varphi_5=x_1^2$. Therefore the Jacobi inversion formulae are as follows.

For $k=4$, $i=1$, we have
\begin{gather*}
\wp_{1,1}\big(u^{[4]}\big)=\frac{
\left|\begin{matrix}1 & x_1(P_1) & x_3(P_1) & x_1^2(P_1) \\ 1 & x_1(P_2) & x_3(P_2) & x_1^2(P_2) \\
1 & x_1(P_3) & x_3(P_3) & x_1^2(P_3) \\ 1 & x_1(P_4) & x_3(P_4) & x_1^2(P_4)
\end{matrix}\right|}
{\left|\begin{matrix}1 & x_1(P_1) & x_3(P_1) & x_2(P_1) \\ 1 & x_1(P_2) & x_3(P_2) & x_2(P_2) \\
1 & x_1(P_3) & x_3(P_3) & x_2(P_3) \\ 1 & x_1(P_4) & x_3(P_4) & x_2(P_4)
\end{matrix}\right|}.
\end{gather*}

For $k=3$, $i=2$, we have
\begin{gather*}
\frac{\sigma_2\big(u^{[3]}\big)}{\sigma_1\big(u^{[3]}\big)}=-\frac{
\left|\begin{matrix}1 & x_1(P_1) & x_2(P_1) \\ 1 & x_1(P_2) & x_2(P_2) \\
1 & x_1(P_3) & x_2(P_3)
\end{matrix}\right|}
{\left|\begin{matrix}1 & x_1(P_1) & x_3(P_1) \\ 1 & x_1(P_2) & x_3(P_2) \\
1 & x_1(P_3) & x_3(P_3)
\end{matrix}\right|}.
\end{gather*}

For $k=2$, we have
\begin{gather*}
\frac{\sigma_3\big(u^{[2]}\big)}{\sigma_2\big(u^{[2]}\big)}=\frac{x_3(P_1)-x_3(P_2)}{x_1(P_2)-x_1(P_1)},\qquad \frac{\sigma_4\big(u^{[2]}\big)}{\sigma_2\big(u^{[2]}\big)}=\frac{x_1(P_1)x_3(P_2)-x_1(P_2)x_3(P_1)}{x_1(P_2)-x_1(P_1)}.
\end{gather*}

For $k=1$, we have
\begin{gather*}
\frac{\sigma_4\big(u^{[1]}\big)}{\sigma_3\big(u^{[1]}\big)}=-x_1(P_1).
\end{gather*}

\section[Vanishing of $\sigma_i$]{Vanishing of $\boldsymbol{\sigma_i}$}\label{section9}

In \cite{Aya2, NY}, the vanishing and the expansion of the sigma functions of $(n,s)$ curves and telescopic curves on the Abel--Jacobi image are studied.
In this section, we show that from Theorem~\ref{main3} we can derive some new vanishing properties of~$\sigma_i$ for telescopic curves immediately.

\begin{Corollary}\label{9}
If $g-a\le k\le g-1$ and $i\ge g-k$, then $\sigma_i\big(u^{[k]}\big)$ does not vanish identically with respect to $P_1,\dots,P_k$.
\end{Corollary}

\begin{proof}
For $i\ge g-k$, $\mu_{k,g+1-i}(P_1,\dots,P_{k})$ does not vanish identically with respect to $P_1,\dots,P_k$.
Therefore Corollary~\ref{9} follows from Theorem~\ref{main3}.
\end{proof}

For $g-a\le k\le g-1$ and $i>g-k$, we consider the expansion
\begin{gather*}
\sigma_{g-k}\left(u^{[k-1]}+\int_{\infty}^{P_k}du\right)=C_k\big(u^{[k-1]}\big)z_k^{\alpha_k}+O\big(z_k^{\alpha_k+1}\big)
\end{gather*}
and
\begin{gather*}
\sigma_i\left(u^{[k-1]}+\int_{\infty}^{P_k}du\right)=C_{k,i}\big(u^{[k-1]}\big)z_k^{\beta_{k,i}}+O\big(z_k^{\beta_{k,i}+1}\big),
\end{gather*}
where $C_k\big(u^{[k-1]}\big)$ and $C_{k,i}\big(u^{[k-1]}\big)$ do not vanish identically with respect to $P_1,\dots,P_{k-1}$.

\begin{Corollary}\label{55}\quad
\begin{enumerate}\itemsep=0pt
\item[$(i)$] We have $\alpha_k=\beta_{k,i}+N(k+1)-N(k)$. In particular, if $g-a< k\le g-1$ and $i>g-k$, then we have $\beta_{k,i}=0$ and $\alpha_k=N(k+1)-N(k)$.

\item[$(ii)$] We have, as a meromorphic function of $P_1,\dots,P_{k-1}$,
\begin{gather}
\frac{C_{k,i}\big(u^{[k-1]}\big)}{C_k\big(u^{[k-1]}\big)}=(-1)^{k+i-g}\mu_{k-1,g+1-i}(P_1,\dots,P_{k-1}).\label{8}
\end{gather}
\end{enumerate}
\end{Corollary}

\begin{proof}
From Theorem \ref{main3}, we have
\begin{gather*}
\frac{\sigma_i\big(u^{[k]}\big)}{\sigma_{g-k}\big(u^{[k]}\big)}=(-1)^{k+i-g}\mu_{k,g+1-i}(P_1,\dots,P_{k}).
\end{gather*}
Therefore we have
\begin{gather*}
\frac{C_{k,i}\big(u^{[k-1]}\big)z_k^{\beta_{k,i}}+O\big(z_k^{\beta_{k,i}+1}\big)}{C_k\big(u^{[k-1]}\big)z_k^{\alpha_k}+O\big(z_k^{\alpha_k+1}\big)} \\
\qquad{} =(-1)^{k+i-g}\mu_{k-1,g+1-i}(P_1,\dots,P_{k-1})z_k^{N(k)-N(k+1)}+O\big(z_k^{N(k)-N(k+1)+1}\big).
\end{gather*}
Therefore we obtain $\beta_{k,i}-\alpha_k=N(k)-N(k+1)$ and~(\ref{8}). On the other hand, if $g-a< k\le g-1$ and $i>g-k$, then from Corollary~\ref{9} $\sigma_i\big(u^{[k-1]}\big)$ does not vanish identically with respect to $P_1,\dots,P_{k-1}$. Therefore, if $g-a< k\le g-1$ and $i>g-k$, then $\beta_{k,i}=0$.
\end{proof}

\section[Example: $(4,6,5)$-curve]{Example: $\boldsymbol{(4,6,5)}$-curve}\label{section10}

By applying Corollary \ref{9} for the $(4,6,5)$ curves, we have $\sigma_3\big(u^{[1]}\big)\not\equiv0$, $\sigma_4\big(u^{[1]}\big)\not\equiv0$, $\sigma_2\big(u^{[2]}\big)\not\equiv0$, $\sigma_3\big(u^{[2]}\big)\not\equiv0$, $\sigma_4\big(u^{[2]}\big)\not\equiv0$, $\sigma_1\big(u^{[3]}\big)\not\equiv0$, $\sigma_2\big(u^{[3]}\big)\not\equiv0$, $\sigma_3\big(u^{[3]}\big)\not\equiv0$, $\sigma_4\big(u^{[3]}\big)\not\equiv0$.

By applying Corollary \ref{55} for the $(4,6,5)$ curves, we have
\begin{gather*}
\sigma_1\big(u^{[3]}\big)=C_3\big(u^{[2]}\big)z_3+O\big(z_3^2\big), \qquad
\sigma_2\big(u^{[2]}\big)=C_2\big(u^{[1]}\big)z_2+O\big(z_2^2\big), \\
\sigma_3\big(u^{[1]}\big)=C_1z_1^4+O\big(z_1^5\big),
\end{gather*}
where $C_3\big(u^{[2]}\big)\not\equiv0$, $C_2\big(u^{[1]}\big)\not\equiv0$, and $C_1\neq0$.

\appendix

\section{Proof of Lemma \ref{lemma5}}

From (\ref{eq-2-5}), for $2\le i\le m$, we have
\begin{gather*}
\frac{\partial F_i}{\partial y_n}= \begin{cases}-\ell_{i,n}y_1^{\ell_{i,1}}\cdots y_n^{\ell_{i,n}-1}\cdots y_{i-1}^{\ell_{i,i-1}}-\sum j_n\lambda_{j_1,\dots,j_m}^{(i)}y_1^{j_1}\cdots y_n^{j_n-1}\cdots y_m^{j_m}, & 1\le n\le i-1, \\
(d_{i-1}/d_i)y_i^{d_{i-1}/d_i-1}-\sum j_i\lambda_{j_1,\dots,j_m}^{(i)}y_1^{j_1}\cdots y_i^{j_i-1}\cdots y_m^{j_m}, & n=i, \\
-\sum j_n\lambda_{j_1,\dots,j_m}^{(i)}y_1^{j_1}\cdots y_n^{j_n-1}\cdots y_m^{j_m} ,& i+1\le n\le m.
\end{cases}
\end{gather*}
Let $\epsilon_k$ be the coef\/f\/icient of $y_1^{\gamma_1}\cdots y_m^{\gamma_m}$ in $\det G_k(Q)$. Since $\det G_k(Q)$ is homogeneous of degree $\sum\limits_{i=2}^ma_id_{i-1}/d_i-\sum\limits_{i=1}^ma_i+a_k$ and $\sum\limits_{i=1}^ma_i\gamma_i=\sum\limits_{i=2}^ma_id_{i-1}/d_i-\sum\limits_{i=1}^ma_i+a_k$, $\epsilon_k$ does not contain $\big\{\lambda_{j_1,\dots,j_m}^{(i)}\big\}$. Therefore $\epsilon_k$ is the determinant of the $(m-1)\times (m-1)$ matrix obtained by deleting the $k$-th column from the $(m-1)\times m$ matrix~$M$
\begin{gather*}
M:=\left(\begin{matrix}
-\ell_{2,1} & d_1/d_2 & 0 & \cdots & 0 \\
-\ell_{3,1} & -\ell_{3,2} & d_2/d_3 & \cdots & 0 \\
\vdots & \vdots & \vdots & \ddots & \vdots \\
-\ell_{m,1} & -\ell_{m,2} & \cdots & -\ell_{m, m-1} & d_{m-1}/d_m
\end{matrix}\right).
\end{gather*}
By multiplying some elementary matrices on the left, the matrix $M$ becomes
\begin{gather*}
\widetilde{M}=\left(\begin{matrix}
z_2 & d_1/d_2 & 0 & \cdots & 0\\
z_3 & 0 & d_2/d_3 & \cdots & 0\\
\vdots & \vdots & \vdots & \ddots & \vdots \\
z_m & 0 & 0 & \cdots & d_{m-1}/d_m
\end{matrix}\right)
\end{gather*}
for certain $z_2,\dots,z_m\in\mathbb{C}$. For $k=1$, we have
\begin{gather*}
\epsilon_1=\frac{d_1}{d_2}\cdot\frac{d_2}{d_3}\cdots\frac{d_{m-1}}{d_m}=\frac{d_1}{d_m}=a_1.
\end{gather*}
For $k\ge2$, we have
\begin{gather*}
\epsilon_k=(-1)^kz_k\cdot\frac{d_1}{d_2}\cdots\check{\frac{d_{k-1}}{d_k}}\cdots\frac{d_{m-1}}{d_m}=(-1)^kz_k\cdot a_1\frac{d_k}{d_{k-1}},
\end{gather*}
where a check on top of a letter signif\/ies deletion.

Since
\begin{gather*}
M\left(\begin{matrix}
a_1 \\
\vdots \\
a_m
\end{matrix}\right)=
\left(\begin{matrix}
0 \\
\vdots \\
0\end{matrix}\right),
\end{gather*}
we have
\begin{gather*}
\widetilde{M}\left(\begin{matrix}
a_1 \\
\vdots \\
a_m
\end{matrix}\right)=
\left(\begin{matrix}
0 \\
\vdots \\
0\end{matrix}\right).
\end{gather*}
Therefore we have $z_ka_1+(d_{k-1}/d_k)a_k=0$ for $2\le k\le m$. Therefore we have $\epsilon_k=(-1)^{k+1}a_k$.

\subsection*{Acknowledgements}

The author would like to thank Professor Shigeki Matsutani for answering a question on the paper~\cite{Matsutani2} kindly and sending his unpublished paper. The author would like to thank Professor Atsushi Nakayashiki for inviting him the conference ``Curves, Moduli and Integrable Systems'' at Tsuda College and giving valuable discussions. The author would like to thank Professor Masato Okado for the support of travel costs for a presentation at Tsukuba University. The author would like to thank Professor Yoshihiro Onishi for inviting him Meijo University and giving valuable discussions. The author would like to thank the anonymous referees for reading our paper carefully and giving many valuable comments. In particular, the author is deeply grateful for their warm encouragement.

\pdfbookmark[1]{References}{ref}
\LastPageEnding

\end{document}